\documentclass[10pt]{amsart}

\def\b{\beta}
\def\hb{\hat{\beta}}
\def\a{\alpha}

\def\n{\noindent}
\def\mcS{\mathcal{S}}

\newtheorem{theorem}{Theorem}
\newtheorem{lemma}[theorem]{Lemma}
\newtheorem{corollary}[theorem]{Corollary}

\newtheorem{definition}{Definition}
\newtheorem{remark}{Remark}

\author[Pawe{\l} Hitczenko]{Pawe{\l} Hitczenko${}^\dagger$}
 \thanks{$\dagger$ Partially supported by a grant from 
 Simons Foundation (grant \#208766 to Pawe{\l} Hitczenko)}
\address{Department of Mathematics, Drexel University, Philadelphia, 
PA  19104, USA} 
\email{phitczenko@math.drexel.edu}

\author[Amanda Parshall]{Amanda Parshall${}^\dagger$}
\address{Department of Mathematics, Drexel University, Philadelphia, 
PA  19104, USA} 
\email{agp47@drexel.edu}

\title[Distribution of parameters in staircase tableaux]{On the distribution of
  parameters
\\ 
in  random weighted  staircase tableaux
}


\begin{document}
\maketitle
\begin{abstract}
  In this paper, we study staircase tableaux, a combinatorial object
  introduced due to its connections with the asymmetric
  exclusion process (ASEP) and Askey-Wilson polynomials. Due to their
  interesting connections, staircase tableaux have been the object of
  study in many recent papers. More specific to this paper, the
  distribution of various parameters in random staircase tableaux has
  been studied. There have been interesting results on parameters
  along the main diagonal, however, no such results have appeared for
  other diagonals. It was conjectured that the distribution of the
  number of symbols
  along the $k$th diagonal is asymptotically Poisson as $k$ and the
  size of the tableau tend to infinity.  We partially prove this
  conjecture; more specifically we prove it for the second main diagonal.
 \end{abstract}

\section{Introduction}
\label{sec:in}
In this paper, we study staircase tableaux, a combinatorial object
introduced (in \cite{CW1}, \cite{CW2}) due to connections with the asymmetric exclusion
process (ASEP)  and Askey-Wilson polynomials. The
ASEP can be defined as a Markov chain with $n$ sites, with at most
one particle occupying each site. Particles may jump to any
neighboring empty site with rate $u$ to the right and rate $q$ to the
left. Particles may enter and exit at the first site with rates $\a$ and $\gamma$ respectively. Similarly, particles may enter and exit the last site with rates $\delta$ and $\b$. The ASEP is an interesting particle model that has been
studied extensively in mathematics and physics. It has also been studied in many other fields, including computational biology \cite{Bu}, and biochemistry, specifically as a primitive model for protein synthesis \cite{GMP}.
Staircase tableaux were introduced per a connection between the steady state distribution of the ASEP and the generating function for staircase tableaux \cite{CW2}. See Section \ref{DN} for a discussion of the ASEP and its connection with staircase tableaux.

In addition to interest in its own right, the ASEP has been known to
have interesting connections in combinatorics and
analysis. Consequently, staircase tableaux have similarly been
connected to many combinatorial objects and a family of
polynomials. In fact, the generating function for staircase tableaux
has been used to give a formula for the moments of Askey-Wilson
polynomials  \cite{CW2}, \cite{CSSW}. Staircase tableaux have also
inherited many interesting properties from other types of tableaux
(See \cite{ABN}, \cite{CH}, \cite{CN}, \cite{CW2}, \cite{CW4}, \cite{CW3},
\cite{HJ3}, \cite{SW}). We refer to \cite[Table~1]{HJ} for a
description of some of the bijections between the various types of the
tableaux. \\

Due to all these interesting connections, staircase tableaux have been the
object of study in many recent papers. In \cite{D-HH}, a probabilistic
approach was developed and the distributions of various parameters were
studied. In fact, it was shown in \cite{D-HH} that the distribution of
the number of $\a$'s and $\b$'s in a staircase tableau are
asymptotically normal, and distributions regarding the  main diagonal
were also given (see Section \ref{SAB}). In \cite{HJ}, the
distribution of each box in a staircase tableau was given, and it was
conjectured that the distribution of the number of symbols along the $k$th diagonal
is asymptotically Poisson as $k$ and the size of the tableau  tend to infinity. The main result of this
paper is the proof of this in a special case when $k=n-1$. That is, we show
that the distribution of the number of symbols along the second main diagonal is
asymptotically Poisson with parameter $1$, see
Theorem~\ref{thm:symb}. Similarly, we show that the number of $\a$'s
(resp. $\b$'s) along the second main diagonal is
asymptotically Poisson with parameter $1/2$, see
Theorem~\ref{DT} and Corollary~\ref{BC}.

\section{Definitions and Notation} \label{DN}
Staircase tableaux were first introduced in \cite{CW1} and \cite{CW2} as follows:
\begin{definition} \label{DEFST}
A staircase tableau of size n is a Young diagram of shape (n, n-1, ..., 1) such that:
\begin{enumerate}
\item The boxes are empty or contain an $\a$, $\b$, $\gamma$, or $\delta$.
\item All boxes in the same column and above an $\a$ or $\gamma$ are empty.
\item All boxes in the same row and to the left of an $\beta$ or $\delta$ are empty.
\item Every box on the diagonal contains a symbol.
\end{enumerate}
\end{definition}
The rows and columns in a staircase tableau are numbered from $1$ through $n$, beginning with the box in the NW-corner and continuing south and east respectively. Each box is numbered $(i,j)$ where $i,j \in \{ 1, 2, ..., n \}$. Note that $i + j \leq n + 1$. We refer to the collection of boxes $(n-i+1, i)$ such that $i=1,2, ..., n$ as the main diagonal, and the collection of boxes $(n-i, i)$ such that $i=1,2, ..., n-1$ as the second main diagonal.

Following the conventions of \cite{HJ}, $\mcS_{n}$ is the set of all staircase tableaux of size $n$. For a given $S \in \mcS_{n}$, the number of $\a$'s, $\gamma$'s, $\b$'s and $\delta$'s in $S$ are denoted by $N_{\a}, N_{\b}, N_{\gamma},$ and $N_{\delta}$ respectively. The weight of $S$ is the product of all symbols in $S$:
\[
wt(S) = \a^{N_{\a}}\b^{N_{\b}}\gamma^{N_{\gamma}}\delta^{N_{\delta}}.
\]
It was known (see e.g. \cite{CD-H}) that the generating function $Z_{n}(\alpha, \beta, \gamma, \delta) := \sum_{S \in \mcS_{n}} wt(S)$ is equal to the product:
\begin{equation}\label{Z4}
Z_{n}(\a, \b, \gamma, \delta)=\prod^{n-1}_{i=0} (\a + \b + \delta + \gamma + i(\a + \gamma)(\b + \delta)).
\end{equation}

\vspace{1cm}

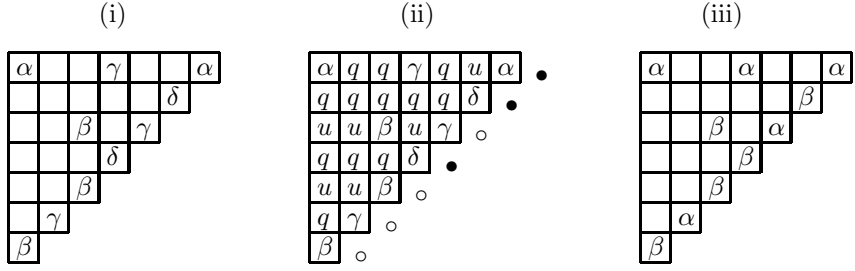
\begin{figure}[htbp] 
\setlength{\unitlength}{0.4cm}
\begin{center}
\begin{picture}
(0,0)(15,8)\thicklines

\put(3,8){(i)} \put(13,8){(ii)}\put(23,8){(iii)}


\put(0,0){\line(0,1){7}}
\put(1,0){\line(0,1){7}}
\put(2,1){\line(0,1){6}}
\put(3,2){\line(0,1){5}}
\put(4,3){\line(0,1){4}}
\put(5,4){\line(0,1){3}}
\put(6,5){\line(0,1){2}}
\put(7,6){\line(0,1){1}}

\put(0,7){\line(1,0){7}}
\put(0,6){\line(1,0){7}}
\put(0,5){\line(1,0){6}}
\put(0,4){\line(1,0){5}}
\put(0,3){\line(1,0){4}}
\put(0,2){\line(1,0){3}}
\put(0,1){\line(1,0){2}}
\put(0,0){\line(1,0){1}}

\put(3.25,6.25){$\gamma$}
\put(6.25, 6.25){$\a$}
\put(5.25, 5.25){$\delta$}
\put(2.25, 4.25){$\b$}
\put(4.25, 4.25){$\gamma$}
\put(3.25, 3.25){$\delta$}
\put(2.25, 2.25){$\b$}
\put(1.25, 1.25){$\gamma$}
\put(0.25, 0.25){$\b$}
\put(0.25, 6.25){$\a$}

\put(10,0){\line(0,1){7}}
\put(11,0){\line(0,1){7}}
\put(12,1){\line(0,1){6}}
\put(13,2){\line(0,1){5}}
\put(14,3){\line(0,1){4}}
\put(15,4){\line(0,1){3}}
\put(16,5){\line(0,1){2}}
\put(17,6){\line(0,1){1}}

\put(10,7){\line(1,0){7}}
\put(10,6){\line(1,0){7}}
\put(10,5){\line(1,0){6}}
\put(10,4){\line(1,0){5}}
\put(10,3){\line(1,0){4}}
\put(10,2){\line(1,0){3}}
\put(10,1){\line(1,0){2}}
\put(10,0){\line(1,0){1}}

\put(16.25, 6.25){$\a$}
\put(15.25, 6.25){$u$}
\put(14.25, 6.25){$q$}
\put(13.25, 6.25){$\gamma$}
\put(12.25, 6.25){$q$}
\put(11.25, 6.25){$q$}
\put(10.25, 6.25){$\a$}
\put(15.25, 5.25){$\delta$}
\put(14.25, 5.25){$q$}
\put(13.25, 5.25){$q$}
\put(12.25, 5.25){$q$}
\put(11.25, 5.25){$q$}
\put(10.25, 5.25){$q$}
\put(10.25, 4.25){$u$}
\put(11.25, 4.25){$u$}
\put(12.25, 4.25){$\b$}
\put(13.25, 4.25){$u$}
\put(14.25, 4.25){$\gamma$}
\put(10.25, 3.25){$q$}
\put(11.25, 3.25){$q$}
\put(12.25, 3.25){$q$}
\put(13.25, 3.25){$\delta$}
\put(10.25, 2.25){$u$}
\put(11.25, 2.25){$u$}
\put(12.25, 2.25){$\beta$}
\put(10.25, 1.25){$q$}
\put(11.25, 1.25){$\gamma$}
\put(10.25, 0.25){$\b$}

\put(17.5,6.0){$\bullet$}
\put(16.5,5.0){$\bullet$}
\put(15.5,4.0){$\circ$}
\put(14.5,3.0){$\bullet$}
\put(13.5,2.0){$\circ$}
\put(12.5,1.0){$\circ$}
\put(11.5,0.0){$\circ$}

\put(21,0){\line(0,1){7}}
\put(22,0){\line(0,1){7}}
\put(23,1){\line(0,1){6}}
\put(24,2){\line(0,1){5}}
\put(25,3){\line(0,1){4}}
\put(26,4){\line(0,1){3}}
\put(27,5){\line(0,1){2}}
\put(28,6){\line(0,1){1}}

\put(21,7){\line(1,0){7}}
\put(21,6){\line(1,0){7}}
\put(21,5){\line(1,0){6}}
\put(21,4){\line(1,0){5}}
\put(21,3){\line(1,0){4}}
\put(21,2){\line(1,0){3}}
\put(21,1){\line(1,0){2}}
\put(21,0){\line(1,0){1}}

\put(24.25,6.25){$\a$}
\put(27.25, 6.25){$\a$}
\put(26.25, 5.25){$\b$}
\put(23.25, 4.25){$\b$}
\put(25.25, 4.25){$\a$}
\put(24.25, 3.25){$\b$}
\put(23.25, 2.25){$\b$}
\put(22.25, 1.25){$\a$}
\put(21.25, 0.25){$\b$}
\put(21.25, 6.25){$\a$}

\end{picture}

\vspace{3cm}
\caption{A staircase tableau of size $7$ with weight $\a^{2}
  \b^{3} \delta^{2} \gamma^{3}$. (ii) The extension of (i) to a
  staircase tableau of weight $\a^{2} \b^{3} \delta^{2}
  \gamma^{3} u^{6} q^{13}$ 
and type $\bullet \bullet\circ  \bullet \circ \circ \circ$. 
(iii) The $\a/\b$-staircase tableau obtained from  (i) by replacing $\gamma$'s with $\a$'s and $\delta$'s with $\b$'s. It's weight is $\a^{5}\b^{5}$. }
\label{F1}
\end{center}

\end{figure}

Notice that there is an involution on the staircase tableaux of a
given size obtained by interchanging the rows and the columns, $\a$'s and $\b$'s, and $\gamma$'s and $\delta$'s, see further \cite{CD-H}. In particular, the fact that $\a$'s and $\b$'s are identical up to this involution allows us to extend results for $\a$'s to results for $\b$'s.

The connection between staircase tableaux and the ASEP requires an
extension of this preceding definition. After following the rules from
Definition \ref{DEFST}, we then fill all the empty boxes with $u$'s
and $q$'s, the rates at which particles in the ASEP jump to the right
and left respectively. We do so by first filling all boxes to the left
of a $\b$ with a $u$ and to the left of a $\delta$ with a $q$. Then,
we fill the empty boxes with a $u$ if it is above an $\a$ or a $\delta$, and $q$ otherwise. The weight of a staircase tableau filled as such is defined in the same way, the product of the parameters in each box. Also, the total weight of all such staircase tableaux, which we denote by $\mcS^{'}_{n}$, is given by:
\[
Z_{n}(\a, \b, \gamma, \delta, q, u) := \sum_{S \in \mcS^{'}_{n}} wt(S).
\]
Then, each staircase tableau of size $n$ is associated with a state of
the ASEP with $n$ sites (See Figure \ref{F1}). This is done by
aligning the Markov chain with the diagonal entries of the staircase tableau. A site is filled if the corresponding diagonal entry is an $\a$ or a $\gamma$ and a site is empty if the corresponding diagonal entry is a $\b$ or a $\delta$. Each staircase tableau's associated state of the ASEP is called its type.

Using this association, it was shown in \cite{CW2} that the steady state probability that the ASEP is in state $\eta$ is given by:
\[
\frac{\sum_{T \in \mathfrak{T}} wt(T)}{Z_{n}}, 
\]
where $\mathfrak{T}$ is the set of all staircase tableau of type $\eta$.

For the purposes of this paper, we will consider 
more simplified
staircase tableaux, namely $\a/\b$-staircase tableaux as
introduced in \cite{HJ}, which are staircase tableaux limited to the
symbols $\a$ and $\b$. The set $\overline{\mcS}_{n} \subset
\mcS_{n}$ denotes the set of all such staircase tableaux. Since the symbols $\a$ and $\gamma$ follow the same rules in the definition, as do $\b$ and $\delta$, any $S \in \mcS_{n}$ can be obtained from an $S^{'} \in \overline{\mcS}_{n}$ by replacing the appropriate $\a$'s with $\gamma$'s and $\b$'s with $\delta$'s. 

The generating function of $\a/\b$-staircase tableaux is:
\[
Z_{n}(\a, \b) := \sum_{S \in \overline{\mcS}_{n}} wt(S) = Z_{n}(\a, \b, 0, 0)
\]
and it follows from $(\ref{Z4})$ that it 
is simply:
\[
Z_{n}(\a, \b) = \a^{n}\b^{n}(a + b)^{\overline{n}}
\]
where $a:= \a^{-1}$ and $b:=\b^{-1}$, a notation that will be used frequently throughout this paper, and $(a+b)^{\overline{n}}$ is the rising factorial of $(a+b)$, i.e. $(a+b)^{\overline{n}} = (a+b)((a+b)+1)\cdots((a+b)+n-1)$.

We wish to consider random staircase tableaux as was done in
\cite{D-HH} but it suffices to study random $\a/\b$-staircase
tableaux, denoted by $S_{n, \a, \b}$, as was done in
\cite{HJ}. All of our results for random $\a/\b$-staircase
tableaux can be extended to random staircase tableaux with all four
parameters, $\a, \gamma, \b, \delta$. This is done by
considering $S_{n, \a + \gamma, \b + \delta}$ and randomly
replacing each $\a$ with $\gamma$ with probability
$\frac{\gamma}{\a + \gamma}$ and similarly, each $\b$ with
$\delta$ with probability $\frac{\delta}{\b + \delta}$
independently for each occurrence. Notice that $Z_{n}(\a, \b, \gamma, \delta) = Z_{n}(\a + \gamma, \b + \delta)$. We also allow all parameters to be arbitrary positive real numbers,  i.e. $\a, \b \in (0, \infty)$, allowing $\a = \infty$ by fixing $\b$ and taking the limit or vice versa, or $\a=\b=\infty$ by taking the limit. The following is a formal definition as in \cite{HJ}:

\begin{definition}
For all $n \geq 1$, $\a, \b \in [0, \infty)$ with $(\a,
\b) \neq (0,0)$, $S_{n,\a,\b}$ is defined to be a random
$\a/\b$-staircase tableau in $\overline{\mcS}_{n}$ with respect
to the probability distribution on $\overline{\mcS}_{n}$ given by:
\[
\forall S \in \overline{\mcS}_{n},\hspace{5mm} \mathbb{P}(S_{n,\a,\b} = S) = \frac{wt(S)}{Z_{n}(\a,\b)} = \frac{\a^{N_{\a}}\b^{N_{\b}}}{Z_{n}(\a,\b)}.
\]
We will also write is as 
\[\mathcal{L}(S_{n,\a,\b})=\frac{\a^{N_\a}\b^{N_\b}}{Z_n(\a,\b)}.\]
\end{definition}
Using this definition, Hitczenko and Janson presented the distribution
of a given box in a random staircase tableau.
If a box is on the main diagonal, the distribution is (see
\cite[Theorem~7.1]{HJ}):
\begin{equation}\label{1BOXD}
\mathbb{P}(S_{n, \a, \b}(i, n+ 1 - i) = \a) = \frac{n-i+b}{n + a + b - 1}.
\end{equation}
Since a box on the main diagonal is never empty, the $\b$ case follows trivially.

If a box is not on the main diagonal, its distribution is (see
\cite[Theorem~7.2]{HJ}):
\begin{eqnarray}  && 
\label{1BOXA}
\mathbb{P}(S_{n, \a, \b}(i, j) = \a) = \frac{j - 1 + b}{(i + j + a + b - 1)(i + j + a + b - 2)} \\ &&
\label{1BOXB}
\mathbb{P}(S_{n, \a, \b}(i, j) = \b) = \frac{i - 1 + a}{(i + j + a + b - 1)(i + j + a + b - 2)}. 
\end{eqnarray}

\section{Subtableaux and Preliminaries}
For an arbitrary $S \in \overline{\mcS}_{n}$ and an arbitrary box $(i,
j)$ in S, define $S[i,j]$ to be the subtableau in
$\overline{\mcS}_{n-i-j+2}$ obtained by deleting the first $i-1$ rows
and $j-1$ columns, see \cite{HJ}. 
The following statement was proven in \cite[Theorem~6.1]{HJ} and is a useful tool in our results:
\begin{equation} \label{SUBT}
S_{n,\alpha, \beta}[i, j] \stackrel{d}{=} S_{n-i-j+2, \hat a, \hat b}, \mbox{ with } \hat{a} = a+i-1 \mbox{ and } \hat{b}=b+j-1.
\end{equation}
The following two lemmas consider the probability of an arbitrary
staircase tableau in $\overline{\mcS}_n$ that is conditioned on having an $\alpha$ or a $\beta$ in the box $(n-1, 1)$. The statements follow almost immediately from the definition of a staircase tableau, but will be used frequently throughout the paper.
\begin{lemma}
\label{SWCorner}
If $S_{n, \a, \b}$ is conditioned on $S_{n, \a, \b}(n-1,
1)= \a$, then the subtableau $S_{n, \a, \b} [1, 3]
\stackrel{d}{=} S_{n-2, \a, \b}$, that is
\[\mathcal{L}(S_{n,\a,\b}\ |\ S_{n,\a,\b}(n-1,1)=\a)=\mathcal{L}(S_{n-2,\a,\b}).\] 
\end{lemma}
\begin{proof}
If $S_{n,\a,\b}$ is a staircase tableau such that the box $S_{n, \a,
  \b}(n-1, 1)= \a$, then the box, $S_{n, \alpha, \beta}(n, 1)=
\b$ and $S_{n, \a, \b}(n-1, 2)= \a$
by the rules of a staircase tableau. The first and second column are otherwise empty by those same rules. The remainder, $S_{n, \a, \b} [1, 3]$,  is an arbitrary staircase tableau of size $n-2$. Therefore, the lemma follows.
\end{proof}

\begin{lemma}
\label{BSWCorner} Let $(S_{n,\alpha,\beta})_{i,j}$ be a tableau
$S_{n,\alpha,\beta}$ with the $i$th row and the $j$th column removed. 
If $S_{n, \alpha, \beta}$ is conditioned on $S_{n, \alpha, \beta}(n-1,
1)= \beta$, then the subtableau $(S_{n, \alpha, \beta})_{n-1, 2}
\stackrel{d}{=} \widetilde{S}_{n-1, \alpha, \beta}$ where $\widetilde{S}_{n-1, \alpha, \beta}$
is random tableau of size $n-1$  conditioned on having a $\b$ in the
$(n-1, 1)$ box. In other words
\[\mathcal{L}(S_{n,\a,\b}\ |\
S_{n,\a,\b}(n-1,1)=\b)=\mathcal{L}(S_{n-1,\a,\b}\ |\ S_{n-1,\a,\b}(n-1,1)=\b).\]
\end{lemma}
\begin{proof}
If $S_{n,\a,\b}$ is a staircase tableau such that the box $S_{n, \alpha,
  \beta}(n-1, 1)= \beta$, then the box $S_{n, \alpha, \beta}(n-1, 2)=
\alpha$ and $S_{n, \alpha, \beta}(n-1,1) = \beta$ by the rules of a
staircase tableau. The second column is otherwise empty by those same
rules. The $n$th row  only has one box, $(n,1)$, which must be a
$\beta$. The remainder is an arbitrary staircase tableau of size $n-1$
conditioned to have a $\beta$ in box $(n-1,1)$. Therefore, the lemma follows.
\end{proof}

\section{Distribution of parameters along the second main diagonal} \label{SAB}
The asymptotic distribution of parameters along the main diagonal is
known. The number of $\alpha/\gamma$ symbols and the number of
$\beta/\delta$ symbols along the main diagonal were proven to be
asymptotically normal in \cite{D-HH}, and the distribution of boxes
along the main diagonal was given in \cite{HJ}. However, the
distributions of parameters on the other diagonals have not been
studied specifically. The expected values were computed in \cite{HJ} and
it was conjectured there that the asymptotic distribution for the
symbols on the $k$th diagonal is Poisson as $k=k_n$ goes to infinity
with $n$ going to infinity. As the first step towards proving that
conjecture we now present results concerning the second main diagonal. In order to simplify notation, let $S_{n, \alpha, \beta}(i)$  be the symbol contained in second main diagonal box $(n-i,i)$ of $S_{n, \alpha, \beta}$. As our first result, the following is the distribution of boxes along the second main diagonal. 

\begin{theorem}
\label{rSymbols}
Let $ 1 \leq j_{1} < ... < j_{r} \leq n - 1$. If 
\begin{equation}\label{2_diff}j_{k} \leq j_{k + 1} - 2, \hspace{2mm} \forall k = 1, 2, ..., r -
1\end{equation} 
then 
\begin{eqnarray*}&&\mathbb{P} (S_{n, \alpha, \beta}(j_{1}) = ... = S_{n, \alpha, \beta}(j_{r}) = \alpha) \\&&\qquad= \prod_{k=1}^{r}\frac{b + j_{r - k + 1} - 2r + 2k - 1}{(a + b + n - 2r + 2k -1)(a + b + n - 2r + 2k - 2)}.\end{eqnarray*}
\noindent (For $r=1$, this is (\ref{1BOXA})). Otherwise, 
\[\mathbb{P} (S_{n, \alpha, \beta}(j_{1}) = ... = S_{n, \a, \b}(j_{r}) = \alpha) = 0.\] 
\end{theorem}

\begin{proof} First note that when (\ref{2_diff}) fails there
  exists $j_{k}$ such that $j_{k} = j_{k+1} - 1$ and thus there must be two $\alpha$'s  in boxes side by side on the $(n-i, i)$ 
diagonal. But this is impossible by the rules of a staircase
tableau as  no symbol  can be put in the diagonal box
$(n-j_k,j_{k+1})$ adjacent to these two boxes. Therefore the
probability is $0$. 

Suppose now that (\ref{2_diff}) holds. We proceed by induction on $r$. By (\ref{SUBT}), $S_{n, \alpha, \beta}[1, j_{1}] \stackrel{d}{=} S_{n -
  j_{1} + 1, \alpha, \hat\beta}$ with $\hat \beta^{-1} = \beta^{-1} + j_{1} - 1$ which yields:
\begin{eqnarray*} &&
\mathbb{P}(S_{n, \alpha, \hat\beta}(j_{1}) = ... = S_{n, \alpha,
  \hat\beta}(j_{r}) = \alpha) 
=\mathbb{P}(S_{n - j_{1} + 1, \alpha, \hat\beta}(1) = ... = S_{n - j_{1} + 1, \alpha, \hat\beta}(j_{r} - j_{1} + 1) = \alpha) \\&&\quad=
\mathbb{P}(S_{n - j_{1} + 1, \alpha, \hat\beta}(j_{2} - j_{1} + 1) = ... = S_{n - j_{1} + 1, \alpha, \hat\beta}(j_{r} - j_{1} + 1) = \alpha \hspace{1mm} | \hspace{1mm} S_{n - j_{1} + 1, \alpha, \hat\beta}(1) = \alpha) \\ & & \qquad \cdot
\mathbb{P}(S_{n - j_{1} + 1, \alpha, \hat\beta}(1) = \alpha).
\end{eqnarray*}
By Lemma \ref{SWCorner} and the induction hypothesis:
\begin{eqnarray*} &&
\mathbb{P}(S_{n - j_{1} + 1, \alpha, \hat\beta}(j_{2} - j_{1} + 1) = ... = S_{n - j_{1} + 1, \alpha, \hat\beta}(j_{r} - j_{1} + 1) = \alpha \hspace{1mm} | \hspace{1mm} S_{n - j_{1} + 1, \alpha, \hat\beta}(1) = \alpha) \\&&\quad=
\mathbb{P} (S_{n - j_{1} - 1, \alpha, \hat\beta}(j_{2} - j_{1} - 1) = ... = S_{n - j_{1} - 1, \alpha, \hat\beta}(j_{r} - j_{1} - 1) = \alpha) \\&&\quad=
\prod_{k=1}^{r-1} \frac{\hat{b} + j_{r-k+1} - j_{1} - 2r + 2k}{(a + \hat{b} + n - j_{1} - 2r + 2k)(a + \hat{b} + n - j_{1} - 2r + 2k - 1)}.
\end{eqnarray*}
By (\ref{1BOXA}):
\[
\mathbb{P}(S_{n - j_{1} + 1, \alpha, \hat\beta}(1) = \alpha) = \frac{\hat{b}}{(n - j_{1} + a + \hat{b})(n - j_{1} + a + \hat{b} - 1)}.
\]
Therefore,
\begin{eqnarray*} &&
\mathbb{P}(S_{n, \alpha, \hat\beta}(j_{1}) = ... = S_{n, \alpha, \hat\beta}(j_{r}) = \alpha) \\&&\quad=
\prod_{k=1}^{r} \frac{\hat{b} + j_{r-k+1} - j_{1} - 2r + 2k}{(a + \hat{b} + n - j_{1} - 2r + 2k)(a + \hat{b} + n - j_{1} - 2r + 2k - 1)} \\&&\quad=
\prod_{k=1}^{r} \frac{b + j_{r-k+1} - 2r + 2k - 1}{(a + b + n - 2r + 2k - 1)(a + b + n - 2r + 2k - 2)}
\end{eqnarray*}
which proves the result.
\end{proof}

Our second main result is the distribution of the number of $\alpha$'s
(and $\beta$'s) along the second main diagonal. The proof requires a  lemma.
\begin{lemma}
\label{LA}
Let 
\[J_{r,m}:=\{ 1 \leq j_{1} < ... < j_{r} \leq m:\ j_{k} \leq j_{k + 1}
- 2, \hspace{2mm} \forall k = 1, 2, ..., r - 1\}.\] 
Then 
\[\sum_{J_{r, m}} \left(\prod_{k=1}^{r} j_{r-k+1}\right)
=\frac{(m+1)_{2r}}{2^{r}r!},\]
where $(x)_r=x(x-1)\dots(x-(r-1))$ is the falling factorial.
\end{lemma}
\begin{proof}
By induction on $r$.
When $r=1$:
\[\sum_{J_{1,m}} \left( \prod_{k=1}^{1} j_{1-k+1} \right) = \sum_{j_1=1}^m j_{1} = \frac{(m+1)m}{2}.
\]
Assume the statement holds for $r-1$. Then:
\begin{eqnarray*}
\sum_{J_{r,m}} \left(\prod_{k=1}^{r} j_{r-k+1} \right) &=&
\sum_{j_r=2r-1}^{m} j_r\left( \sum_{J_{r-1,j_r-2}}  \prod_{k=2}^{r} j_{r-k+1} \right) \\&=&
\sum_{j_r=2r-1}^{m} j_{r}
\frac{(j_r-1)_{2(r-1)}}{2^{r-1}(r-1)!}\\&=&\frac1{2^{r-1}(r-1)!}\sum_{j_{r}=2r-1}^m
( j_{r})_{2r-1}
\end{eqnarray*}
where the second equality is by the induction hypothesis.
Since
\[\sum_{j_r=2r-1}^{m} (j_{r})_{2r-1}=\sum_{j_{r}=0}^m (
j_{r})_{2r-1}\]
the lemma will be proved once we verify that 
\[\sum_{j=0}^m(j)_t
=\frac{(m+1)_{t+1}}{t+1},\]
for any non-negative integer $t$ (and apply it with $t=2r-1$). 
Using the identity 
\[\sum_{j=0}^m{j\choose t}={m+1\choose t+1}\]
(see, e.g. \cite[Formula~(5.10)]{GKP})   we see that 
\[\sum_{j=0}^m(j)_t=
\sum_{j=0}^{m} \frac{j!}{(j-t)!}=t!\sum_{j=0}^m{j\choose t}
=t!{m+1\choose  t+1}
=t!\frac{(m+1)_{t+1}}{(t+1)!}=\frac{(m+1)_{t+1}}{m+1},\]
as asserted.
 \end{proof}

Finally, define $A_{n}$ and $B_{n}$ to be the number of $\alpha$'s and
$\beta$'s on the second main diagonal, 
i.e. 
$A_{n} := \sum^{n-1}_{j=1} I_{S_{n,\alpha, \beta}(j) = \alpha}$ and
$B_{n} := \sum^{n-1}_{j=1} I_{S_{n,\alpha, \beta}(j) = \beta}$. Then,
the asymptotic distribution of $A_n$ and $B_n$ is given in the following theorem and corollary.
\begin{theorem}
\label{DT}
Let $Pois(\lambda)$ be a Poisson random variable with parameter
$\lambda$. Then, as $n\to\infty$,
\begin{equation}
A_{n} \stackrel{d}{\rightarrow} Pois \left( \frac{1}{2} \right).
\end{equation}
\end{theorem}
\begin{proof}
By \cite[Theorem~20, Chapter~1]{B} it suffices to show that the $r$th
factorial moment of 
$A_{n}$ 
satisfies:
\begin{equation}
\mathbb{E}(A_{n})_{r}\to\left(\frac12\right)^r\hspace{5mm}\mbox{as\ }\n\to\infty.
\end{equation}
For the ease of notation let $I_j:=I_{S_{n,\a,\b}(j)=\a}$ and consider 
\begin{eqnarray*} 
z^{A_{n}} &=& z^{\sum^{n-1}_{j=1} I_j} = \prod^{n-1}_{j=1} z^{I_j} = \prod^{n-1}_{j=1} (1 + (z-1))^{I_ j} = \prod^{n-1}_{j=1}(1 + I_j(z-1)) \\&=&
1 + \sum^{n-1}_{r=1} \left( \sum_{1 \leq j_{1} < ... < j_{r} \leq n - 1} \left( \prod^{r}_{k=1} I_{j_k} \right) \right) (z - 1)^{r} \\&=& 1 + \sum^{n-1}_{r=1} (z - 1)^{r} \left( \sum_{1 \leq j_{1} < ... < j_{r} \leq n - 1} \left( \prod^{r}_{k=1} I_{j_{k}} \right) \right).
\end{eqnarray*}
Thus, 
\[
\mathbb{E}(z^{A_{n}}) = 1 + \sum^{n-1}_{r=1} (z - 1)^{r} \left( \sum_{1 \leq
  j_{1} < ... < j_{r} \leq n - 1} 
\mathbb{P}(I_{j_{1}} \cap \ldots \cap
I_{j_{r}}) \right). 
\]
Hence
\[\mathbb{E}(A_{n})_{r} =
\frac{d^r}{dz^r}(\mathbb{E}z^{A_{n}}) |_{z=1} =
r! \left( \sum_{1 \leq j_{1} < ... < j_{r} \leq n - 1} \mathbb{P}(I_{j_{1}} \cap \ldots \cap I_{j_{r}}) \right). 
\]
By Theorem~\ref{rSymbols} and Lemma~\ref{LA}
\begin{eqnarray*} 
\mathbb{E}(A_{n})_{r}& =& r! 
\sum_{J_{r,n-1}} \left( \prod_{k=1}^{r} \frac{b + j_{r-k+1} - 2r + 2k - 1}{(a + b + n - 2r + 2k - 1)(a + b + n - 2r + 2k - 2)} \right) \\ &\approx&
r!\sum_{J_{r, n - 1}} \left( \prod_{k=1}^{r} \frac{j_{r-k+1}}{n^{2}}
\right) = \frac{r!}{n^{2r}} \frac{(n)_{2r}}{2^{r}r!} \rightarrow
\left(
 \frac{1}{2} \right)^{r},\quad\mbox{as\ }n \rightarrow \infty. 
\end{eqnarray*}
\end{proof} 
\begin{corollary}
\label{BC}
The $r$th factorial moment of the number $B_{n}$ of $\b$'s on the second main diagonal of a random staircase tableau of size $n$ satisfies:
\begin{equation}
\mathbb{E}(B_{n})_{r}\to\left(\frac12\right)^r\quad\mbox{as\ } n\to\infty.
\end{equation}
Furthermore,
\begin{equation}
B_{n} \stackrel{d}{\rightarrow} Pois \left( \frac{1}{2} \right)\quad\mbox{as\ } n\to\infty.
\end{equation}\end{corollary}
\begin{proof}
This follows by symmetry, see Section~\ref{DEFST}.
\end{proof}
\begin{remark}
Theorem~\ref{DT} and Corollary~\ref{BC} hold regardless of the values
of $\a$ and $\b$ including the cases discussed earlier when
$\a=\infty$, $\b=\infty$, or $\a=\b=\infty$. As noted in
\cite[Examples~3.6 and 3.7]{HJ} these cases correspond to staircase
tableaux with the maximal number of $\a$'s (or $\b$'s) and the maximal
number of symbols, respectively. 
\end{remark}

\section{Distribution of Non-Empty Boxes}
The number of $\a$'s and the number of $\b$'s, $N_{a}$ and $N_{b}$,
are not independent random variables, and the second main diagonal may
have empty boxes. Therefore, in order to completely describe the
second main diagonal, we must consider both symbols collectively. First, we present the distribution of non-empty boxes along the second main diagonal.
\begin{theorem}
\label{rNZero}
Let $ 1 \leq j_{1} < ... < j_{r} \leq n - 1$. If (\ref{2_diff}) holds 
then 
\[\mathbb{P} (S_{n, \a, \b}(j_{1}) \neq 0, ... , S_{n, \a, \b}(j_{r}) \neq 0) = \prod_{k=1}^{r}\frac{1}{(n + a + b - r + k - 1)}.\]
(For $r=1$, this is obtained by adding (\ref{1BOXA}) and (\ref{1BOXB})).  
Otherwise, 
\[\mathbb{P} (S_{n, \a, \b}(j_{1}) \neq 0, ... , S_{n, \a, \b}(j_{r}) \neq 0) = 0.\] 
\end{theorem}
\begin{proof}
Suppose (\ref{2_diff}) holds. We proceed by induction on $r$.

By (\ref{SUBT}), $S_{n, \a, \b}[1, j_{1}] \stackrel{d}{=} S_{n - j_{1} + 1, \a, \hb}$ with $\hb^{-1} = \b^{-1} + j_{1} - 1$ which yields:
\begin{eqnarray*} &&
\mathbb{P}(S_{n, \a, \hb}(j_{1}) \neq 0, ... , S_{n, \a, \b}(j_{r}) \neq 0) \\ &&\quad=
\mathbb{P}(S_{n - j_{1} + 1, \a, \hb}(1) \neq 0, ... , S_{n - j_{r} + 1, \a, \hb}(j_{r} - j_{1} + 1) \neq 0).
\end{eqnarray*}
Further
\begin{eqnarray*} &&
\mathbb{P}(S_{n - j_{1} + 1, \a, \hb}(1) 
\neq 0, ... , S_{n - j_{r} + 1, \a, \hb}(j_{r} - j_{1} + 1) \neq 0) \\
&&\quad=
\mathbb{P}(S_{n - j_{1} +
 1, \a, \hb} (1) = \a ,S_{n - j_{1} + 1, \a, \hb}(j_{2} - j_{1} + 1) \neq 0, ... ,
S_{n - j_{1} + 1, \a, \hb}(j_{r} - j_{1} + 1) \neq 0) 
\\&&\qquad+
\mathbb{P}(S_{n - j_{1} + 1, \a, \hb} (1) = \b ,S_{n - j_{1} + 1, \a, \hb}(j_{2} - j_{1} + 1) \neq 0, ... ,
S_{n - j_{1} + 1, \a, \hb}(j_{r} - j_{1} + 1) \neq 0 )
\\ &&\quad=
\mathbb{P}(S_{n - j_{1} + 1, \a, \hb}(j_{2} - j_{1} + 1) \neq 0, ... ,
S_{n - j_{1} + 1, \a, \hb}(j_{r} - j_{1} + 1) \neq 0) \hspace{1mm} |
\hspace{1mm} S_{n - j_{1} + 1, \a, \hb} (1) = \a)
\\ &&\qquad \cdot
\mathbb{P} (S_{n - j_{1} + 1, \a, \hb} (1) = \a)\\&&\quad+
\mathbb{P}(S_{n - j_{1} + 1, \a, \hb}(j_{2} - j_{1} + 1) \neq 0, ... , S_{n - j_{1} + 1, \a, \hb}(j_{r} - j_{1} + 1) \neq 0 \hspace{1mm} | \hspace{1mm} S_{n - j_{1} + 1, \a, \hb} (1) = \b) \\ &&\qquad\cdot
\mathbb{P} (S_{n - j_{1} + 1, \a, \hb} (1) = \b).
\end{eqnarray*}
Now consider two cases: \\
\textbf{Case 1:} $S_{n - j_{1} + 1, \a, \hb}(1) = \a$.
By Lemma \ref{SWCorner} and the induction hypothesis:
\begin{eqnarray*} &&
\mathbb{P}(S_{n - j_{1} + 1, \a, \hb}(j_{2} - j_{1} + 1) \neq 0, ... , S_{n - j_{1} + 1, \a, \hb}(j_{r} - j_{1} + 1) \neq 0) \hspace{1mm} | \hspace{1mm} S_{n - j_{1} + 1, \a, \hb} (1) = \a) \\ &&\quad= 
\mathbb{P}(S_{n - j_{1} - 1, \a, \hb}(j_{2} - j_{1} - 1) \neq 0, ... , S_{n - j_{1} - 1, \a, \hb}(j_{r} - j_{1} - 1) \neq 0) \\ &&\quad= \prod_{k=1}^{r-1}\frac{1}{(n - j_{1} + a + \hat{b} - r + k - 1)}.
\end{eqnarray*}
Also, by (\ref{1BOXA}), 
\[
\mathbb{P} (S_{n - j_{1} + 1, \a, \hb} (1) = \a) = \frac{\hat{b}}{(n - j_{1} + a + \hat{b})(n - j_{1} + a + \hat{b} - 1)}.
\]
Therefore,  
\begin{eqnarray*} &&
\mathbb{P}(S_{n - j_{1} + 1, \a, \hb}(1) = \a,  S_{n - j_{1} + 1, \a, \hb}(j_{2} - j_{1} + 1) \neq 0, ... , S_{n - j_{r} + 1, \a, \hb}(j_{r} - j_{1} + 1) \neq 0) \\ &&\quad=
\frac{\hat{b}}{(n - j_{1} + a + \hat{b})(n - j_{1} + a + \hat{b} - 1)} \cdot \prod_{k=1}^{r-1}\frac{1}{(n - j_{1} + a + \hat{b} - r + k - 1)}.
\end{eqnarray*}

 \textbf{Case 2:} $S_{n - j_{1} + 1, \a, \hb}(1) = \b$. By Lemma~\ref{BSWCorner} 
\begin{eqnarray*} &&
 \mathbb{P}(S_{n - j_{1} + 1, \a, \hb}(j_{2} - j_{1} + 1) \neq 0, ... , S_{n - j_{1} + 1, \a, \hb}(j_{r} - j_{1} + 1) \neq 0 \hspace{1mm} | \hspace{1mm} S_{n - j_{1} + 1, \a, \hb} (1) = \b) \\ &&\quad=
\mathbb{P} (S_{n-j_{1}, \a, \hb}(j_{2} - j_{1}) \neq 0, ... , S_{n - j_{1r}, \a, \hb}(j_{r} - j_{1}) \neq 0  \hspace{1mm} | \hspace{1mm} S_{n-j_{1}, \a, \hb} (n-j_{1}, 1) = \b) \\ &&\quad= 
\frac
{\mathbb{P}(S_{n-j_{1}, \a, \hb}(j_{2} - j_{1}) \neq 0, ... , S_{n -
    j_{1}, \a, \hb}(j_{r} - j_{1}) \neq 0, \hspace{1mm} S_{n-j_{1},
    \a, \hb} (n-j_{1}, 1) = \b)}{\mathbb{P}(S_{n-j_{1}, \a,
    \hb}(n-j_{1}, 1) = \b)}. 
\end{eqnarray*}
The numerator is equal to 
\begin{eqnarray*}&&\mathbb{P}(S_{n-j_{1}, \a, \hb}(j_{2} - j_{1}) \neq 0, ... , S_{n - j_{1}, \a, \hb}(j_{r} - j_{1}) \neq 0) \\ &&\quad- 
\mathbb{P}(S_{n-j_{1}, \a, \hb}(j_{2} - j_{1}) \neq 0, ... , S_{n - j_{1}, \a, \hb}(j_{r} - j_{1}) \neq 0, \hspace{1mm} S_{n-j_{1}, \a, \hb} (n-j_{1}, 1) = \a)  \\  &&\quad=
\mathbb{P}(S_{n-j_{1}, \a, \hb}(j_{2} - j_{1}) \neq 0, ... , S_{n - j_{1} , \a, \hb}(j_{r} - j_{1}) \neq 0) \\ &&\quad-
\mathbb{P}(S_{n-j_{1}, \a, \hb}(j_{2} - j_{1}) \neq 0, ... , S_{n - j_{1} , \a, \hb}(j_{r} - j_{1}) \neq 0 \hspace{1mm} | \hspace{1mm} S_{n-j_{1}, \a, \hb} (n-j_{1}, 1) = \a) \\ &&\qquad\cdot
\mathbb{P}(S_{n-j_{1}, \a, \hb} (n-j_{1}, 1) = \a).
\end{eqnarray*}
By \cite[Lemma~7.5]{HJ}  and the induction hypothesis 
the conditional probability above is
\begin{equation} \label{4b}
\mathbb{P}(S_{n-j_{1} - 1, \a, \hb} (j_{2} - j_{1} - 1) \neq 0, ... , S_{n - j_{1} - 1, \a, \hb}(j_{r} - j_{1} - 1) \neq 0) = \prod_{k=1}^{r-1} \frac{1}{n - j_{1} + a + \hat{b} - r + k - 1}. 
\end{equation}
By (\ref{1BOXB}), (\ref{1BOXD}) (and some algebra), the induction
hypothesis, and (\ref{1BOXA}), respectively, 
\begin{equation} \label{1b}
\mathbb{P} (S_{n - j_{1} + 1, \a, \hb} (1) = \b) = \frac{n - j_{1} + a - 1}{(n - j_{1} + a + \hat{b})(n - j_{1} + a + \hat{b} - 1)}
\end{equation}
\begin{equation} \label{2b}
\frac{1}{\mathbb{P}(S_{n-j_{1}, \a, \hb}(n-j_{1}, 1) = \beta)}  = \frac{n - j_{1} + a + \hat{b} - 1}{n - j_{1} + a - 1} 
\end{equation}
\begin{equation} \label{3b}
\mathbb{P}(S_{n-j_{1}, \a, \hb}(j_{2} - j_{1}) \neq 0, ... , S_{n -
  j_{1} + 1, \a, \hb}(j_{r} - j_{1}) \neq 0) = \prod_{k=1}^{r-1}
\frac{1}{n - j_{1} + a + \hat{b} - r + k 
} 
\end{equation}
\begin{equation}\label{5b}
\mathbb{P}(S_{n-j_{1}, \a, \hb} (n-j_{1}, 1) = \a) = \frac{\hat{b}}{(n - j_{1} + a + \hat{b} - 1)}.
\end{equation}
Combining (\ref{4b}) - (\ref{5b}): 
\begin{eqnarray*} &&
\mathbb{P}(S_{n - j_{1} + 1, \a, \hb}(j_{2} - j_{1} + 1) \neq 0, ... , S_{n - j_{1} + 1, \a, \hb}(j_{r} - j_{1} + 1) \neq 0, S_{n - j_{1} + 1, \a, \hb} (1) = \b) \\ &&\quad=
\frac{1}{n - j_{1} + a + \hat{b}} \cdot \Big( \prod_{k=1}^{r-1}
\frac{1}{n - j_{1} + a + \hat{b} - r + k 
} \\ &&\qquad-
\frac{\hat{b}}{n -j_{1} + a + \hat{b} - 1
} \prod_{k=1}^{r-1} \frac{1}{n - j_{1} + a + \hat{b} - r + k - 1} \Big).
\end{eqnarray*}
Adding Case 1 and Case 2: 
\begin{eqnarray*} &&
\mathbb{P} (S_{n, \a, \hb}(j_{1}) \neq 0, ... , S_{n, \a, \hb}(j_{r})
\neq 0) =
\frac{1}{n - j_{1} + a + \hat{b}} \cdot \prod_{k=1}^{r-1} \frac{1}{n -
  j_{1} + a + \hat{b} - r + k 
} \\ &&\quad=
\prod_{k=1}^{r} \frac{1}{n - j_{1} + a + \hat{b} - r + k} 
=
\prod_{k=1}^{r} \frac{1}{n + a + b - r + k - 1}
\end{eqnarray*}
which proves our assertion  when (\ref{2_diff}) holds.

If  there exists $j_{k}$ such that $j_{k} = j_{k+1} - 1$, then
$\{S_{n, \a, \b}(j_{1}) \neq 0, ... , S_{n, \a, \beta}(j_{r}) \neq 0\}$ implies  that two boxes side by side on the $(n-i, i)$ diagonal are non-empty, which is impossible by the rules of a  staircase tableau. Therefore the probability is $0$.
\end{proof}

As our final result, we consider the number of symbols on the second
main diagonal, which we denote by $X_{n}$. Then $X_{n} =
\sum^{n-1}_{j=1} I_{j}$ where $I_{j} := I_{S_{n,\alpha, \beta}(j) \neq
  0}$. The asymptotic distribution of the number of symbols on the second main diagonal is given by:
\begin{theorem}\label{thm:symb} As $n\to\infty$, 
\begin{equation} \label{nzp2}
X_{n} \stackrel{d}{\rightarrow} Pois \left( 1 \right).
\end{equation}
\end{theorem}

\begin{proof}
By Theorem~\ref{rNZero}  and the same argument as in Theorem~\ref{DT} 
\begin{eqnarray*}  
\mathbb{E}(X_{n})_{r} &=& r! \sum_{1 \leq j_{1} < ... < j_{r} \leq n - 1} \mathbb{P}(I_{j_{1}} \cap \ldots \cap I_{j_{r}})  
\\&=& r!|J_{r,n-1}|\prod_{k=1}^{r}\frac{1}{n + a + b - r + k -1}
\\ &=&r!\left( {n-1\choose r}+O(n^{r-1})\right) \prod_{k=1}^{r}\frac{1}{n + a + b - r + k -1}\\&
\approx& 
r! \cdot \frac{(n-1)^{r}}{r!n^{r}} 
\rightarrow 1 \quad \mbox{as\ } n \rightarrow \infty.
\end{eqnarray*}
The result  then follows by \cite[Theorem~20]{B}, as discussed in the proof of Theorem~\ref{DT}.
\end{proof}


\end{document}